\documentclass[10pt]{amsart}
\usepackage{amssymb}

 \newtheorem{thm}{Theorem}[section]
 \newtheorem{cor}[thm]{Corollary}
 \newtheorem{lem}[thm]{Lemma}
 \theoremstyle{definition}
 \newtheorem{defn}[thm]{Definition}
 \newtheorem{nta}[thm]{Notation}

\newcommand{\ord}{{\mathrm{o}}}
\newcommand{\pg}{{\mathcal{P}}}
\newcommand{\dpg}[1]{\overrightarrow{\pg}(#1)}
\newcommand{\des}[1]{\overrightarrow{E}(#1)}
\newcommand{\bes}[1]{\overleftrightarrow{E}(#1)}
\newcommand{\cyc}{\text{C}}

\begin{document}

\title
{A group sum inequality and its application to power graphs}

\author{Brian Curtin}
\author{G.~R.~Pourgholi}

\address{{\bf Brian Curtin}: Department of Mathematics and Statistics, University of South Florida, Tampa FL, 33620}
\email{bcurtin@usf.edu}

\address{{\bf G. R. Pourgholi}: School of Mathematics, Statistics and Computer Science,
University of Tehran, Tehran 14155-6455, I.~R.~Iran}

\email{pourgholi@ut.ac.ir}

\begin{abstract}
Let $G$ be a finite group of order $n$, and let $\cyc_n$ be the cyclic group of order $n$.  We show that 
$\sum_{g \in \cyc_n} \phi(\ord(g))\geq \sum_{g \in G} \phi(\ord(g))$,
with equality if and only if $G$ is isomorphic to $\cyc_n$.
As an application, we show that among all finite groups of a given order, the cyclic group of that order has the maximum number of undirected
edges in its directed power graph.

\end{abstract}

\maketitle


\section{Introduction}

Our main result is a  group theoretic inequality, which we apply to power graphs.

\begin{defn}
\label{def:phi(G)}
Let $G$ be a finite group. 
For  $g\in G$,  let $\ord(g)$ denote the order of $g$. 
Let  $\phi$ denote the Euler totient function.
Define
\begin{equation}
\label{eq:newsumgroup}
    \phi(G) = \sum_{g \in G} \phi(\ord(g)).
\end{equation}
\end{defn}

\begin{thm}[Main Theorem]
\label{thm:main-restated}
Let $G$ be a finite group of order $n$, and let $\cyc_n$ be the cyclic group of order $n$. Then 
\begin{eqnarray}
\phi(\cyc_n)     & \geq  & \phi(G),       \label{eq:phiZ>phiG}
\end{eqnarray}
with equality if and only if $G$ is isomorphic to $\cyc_n$.
\end{thm}

Our motivation for (\ref{eq:phiZ>phiG})  lies in our interest in power graphs of finite groups.  

\begin{defn}
\label{def:dpg-ue}
The {\em directed power graph} $\dpg{G}$ of a 
group $G$ has vertex set $G$ and directed edge set
  $\des{G} = \{(g,h)\,|\, g, h\in G,\,  h\in\langle g
  \rangle\backslash\{g\}\}$.
The set of {\em undirected edges}  of 
$\dpg{G}$ is   $\bes{G}
           =\{\{g,h\}\,|\, (g,h), (h,g)\in \des{G} \}$.   
 \end{defn}

Power graphs are among the various graphs related to algebraic structures.  They were introduced in \cite{KQ1, k, KQ2, KQS} in connection with groups and semigroups.  For more information about power graphs, the reader is referred to the survey \cite{AbKeCh}, which contains a full review of the literature to  date. 
 From Definition \ref{def:dpg-ue}, we immediately get the following.

\begin{lem}
In the directed power graph of a group, there is a pair of oppositely directed edges between two distinct group elements precisely when they generate the same subgroup.  
\end{lem}

\begin{cor}
\label{cor:besGinphi}
With reference to Definition \ref{def:phi(G)}, 
$g\in G$ is a vertex in ($\phi(\ord(g))-1$)-many 
undirected edges of $\dpg{G}$.  In particular, 
\begin{equation}
\label{eq:biEcount}
  |\bes{G}|
             = \frac{1}{2} \sum_{g\in G}( \phi(\ord(g))-1 ) =\frac{\phi(G)-|G|}{2}.
\end{equation}
\end{cor}

It was shown in \cite{a} that among directed power graphs of groups of a given finite order, that  of the cyclic group has the maximum number of edges.  In  \cite{BCGRP:maxedge},  we showed that the same is true for undirected power graphs.    In light of Corollary \ref{cor:besGinphi},
Theorem \ref{thm:main-restated} is equivalent to the following related result.

\begin{thm}
\label{thm:main} 
Among all groups of a given finite order, the cyclic
group of that order has the maximum number of undirected edges
in its directed power graph. 
\end{thm}

\section{A criterion for a normal cyclic Sylow subgroup}

We develop a criterion for the existence of a cyclic normal Sylow subgroup.

\begin{nta}
\label{nta:nfactored} Let $n$ be a positive integer. Write
    $n = p_1^{\alpha_1}p_2^{\alpha_2}\cdots p_k^{\alpha_k}$
for primes
     $p_1 < p_2 < \cdots < p_k$
and positive integers $\alpha_1$, $\alpha_2$, \ldots, $\alpha_k$.
Abbreviate $p=p_k$ and $\alpha=\alpha_k$.     
Let  
\begin{equation}
\label{eq:Qdef}
    Q = \prod_{h=1}^{k}\frac{p_h+1}{p_h-1}.
\end{equation}
\end{nta}

An elementary exercise in the same vein as \cite[p.~143, exercise 5]{burton:ENT}  gives two expressions  for $\phi(\cyc_n)$ derived from $n$ (see also \cite[Lemma 2.5]{BCGRP:maxedge}).

\begin{lem}
\label{lem:phiCn}
With Notation \ref{nta:nfactored}, let $\cyc_n$ be the cyclic group of order $n$. Then    
\begin{equation}
\label{eq:PhiZn}
    \phi(\cyc_n) = \sum_{d|n}\phi(d)^{2}  
       = \prod_{h=1}^{k}  \frac{p_h^{2\alpha_h}(p_h-1)+2}{p_h+1}.
 \end{equation}    
\end{lem}

Subtracting the 2 from the numerator of each factor of (\ref{eq:PhiZn}) gives the lower bound
\begin{equation}
\label{eq:phiC>n^2/Q}
        \phi(\cyc_n)> n^2/ Q.
\end{equation}
We may write
\begin{equation}
\label{eq:sumphiineq-rhs}
  Q
    = \frac{1}{p_{1}-1}
             \left( \frac{p_{1}+1}{p_{2}-1}
                   \cdots \frac{p_{k-2}+1}{p_{k-1}-1}\frac{p_{k-1}+1}{p_{k}-1}  
                                                                                 \right) 
           (p_{k}+1). 
\end{equation}
Observe that if $(p_{h-1}, p_{h})\not= (2,3)$, then for $1\leq h \leq k$
\begin{equation}
\label{eq:pipi-1}
  \frac{p_{h-1}+1}{p_{h}-1}\leq 1.
\end{equation}
This immediately gives

\begin{lem}
With Notation \ref{nta:nfactored}, assume $n$ is odd.
Then 
\begin{equation}
\label{eq:Q.p+1/p1-1}
  Q    \leq  \frac{p+1}{p_{1}-1}. 
\end{equation}
\end{lem}

In Table \ref{tab:specialQ} we record data concerning some  sets of primes which require special treatment.  Let $\pi(i)$ denote the $i^{\mathrm{th}}$ prime number.  For each positive integer $\ell$, let 
   $\mathcal{F}_\ell =\{ \pi(i)\, |\,  1\leq i \leq \ell\}$ and
   $\mathcal{S}_\ell=\{ \pi(i)\, |\,  1\leq i \leq \ell-1\}\cup\{\pi(\ell+1)\}$.  
Write $Q(\mathcal{X})$ to denote the value of $Q$ when the set of distinct prime factors of $n$ is   $\mathcal{X}$.  
 \begin{table}[h]
\[\begin{array}{|l|c|c|c|c|c|c|c|c|c|}
\hline
\ell &1&2&3&4&5&6& 7& 8&9\\
\hline
\pi(\ell)& 2& 3& 5& 7& 11& 13& 17& 19&23\\
\hline
Q(\mathcal{F}_\ell) & 3& 6& 9& 12& 72/5& 84/5& 189/10& 21&252/11\\
\hline
Q(\mathcal{S}_\ell) &  2 & 9/2 & 8 & 54/5& 14 & 81/5 & 56/3 & 1134/55 & *\\
\hline
\end{array}\]
 \caption{Some special values of $Q$}
\label{tab:specialQ}
\end{table}

\begin{lem}
\label{lem:Q<=p}
With Notation \ref{nta:nfactored},   the following hold.
\begin{enumerate}
\item \label{lem:numbertheoryresult}
Suppose that either $k\geq 9$ or  
     $\{ p_i\,|\, 1\leq i \leq k\} \not =\mathcal{F}_k$.  
Then $Q  \leq {p+1}$.
\item \label{lem:oddnumbers} 
      Suppose $n$ is odd.  Then $Q < p$.
\end{enumerate}
\end{lem}

\begin{proof}
(i): The excluded sets of prime factors are those in Table \ref{tab:specialQ} with $\ell <9$.  The inequality fails for the first 8 values of $\mathcal{F}_\ell$ but holds  for the $9^{\mathrm{th}}$. From Table \ref{tab:specialQ} we also see that the inequality holds when the set of prime factors of $n$ is $\mathcal{S}_k$ for $1\leq k \leq 8$. 
Referring to (\ref{eq:sumphiineq-rhs}), Equation (\ref{eq:pipi-1}) gives that 
the sequence  $(p_{i-1}+1)/(p_i-1)$ is nondecreasing (except when $(p_{1},p_{2})\not=(2,3)$), so once the inequality is satisfied by an initial subset of prime factors it is satisfied thereafter. Moreover, replacing a prime with a larger prime also preserves the inequality.  The result follows.

(ii):  By (\ref{eq:Q.p+1/p1-1}), and since $p_1$, $p\geq 3$, we have
   $Q\leq {(p + 1)}/{(p_{1} - 1)}  \leq  {(p + 1)}/{2}< p$,
as required.
\end{proof}

It is well-known that
\begin{equation}
\label{eq:phi(n)primes} 
\phi(n) =
   p_1^{\alpha_1-1}(p_1-1)p_2^{\alpha_2-1}(p_2-1)
                                     \cdots p_k^{\alpha_k-1}(p_k-1).
\end{equation}
Immediate consequences include the following:
\begin{eqnarray}
\label{eq:n=phinx}
   n &=& \phi(n) \cdot \frac{p_1}{p_1-1} \cdot  \frac{p_2}{p_2-1}
                 \cdots  \frac{p_k}{p_k-1},\\
\label{eq:a|b=>phia|phib}    
  a|b &\Rightarrow & \phi(a) | \phi(b).               
\end{eqnarray}

\begin{lem}
\label{lem:n/Qgeqph(n/p)p}
With Notation \ref{nta:nfactored},
suppose that $n \neq 2^{\alpha}$ for any $\alpha \geq 0$. 
Then
\begin{equation}
\label{eq:n=phinx11} 
    n \geq   Q\phi(\frac{n}{p^\alpha})
p^{\alpha-1},
\end{equation}
with equality if and only if $n = 2^{\alpha}3^{\beta}$ and
$\alpha$, $\beta > 0$.
\end{lem}

\begin{proof}
If $n=p^\alpha$, then (\ref{eq:n=phinx11}) become 
   $p^\alpha\geq p^{\alpha-1}(p+1)/(p-1)$, 
which holds strictly since $p\neq 2$. The inequality fails if $n=2^\alpha$.   Now suppose that $n$ has at least two distinct prime factors. 
By (\ref{eq:Qdef}) and (\ref{eq:n=phinx}),
\[
\frac{n}{Q}
       = \phi(n)  p_{1}
      \frac{p_2}{(p_{1} + 1)}\frac{p_3}{(p_{2}+1)} \cdots
      \frac{p}{(p_{k-1}+1)}\frac{1}{(p+1)},
\]
By (\ref{eq:phi(n)primes}), $\phi(n) = \phi({n}/{p^{\alpha}})
p^{\alpha-1}(p-1)$, so
\[
\frac{n}{Q}
       = \phi(\frac{n}{p^{\alpha}})p^{\alpha - 1}(p - 1).\frac{p_{1}}{(p + 1)}
       \left(\frac{p_2}{(p_{1} + 1)}\frac{p_3}{(p_{2}+1)} \cdots
      \frac{p}{(p_{k-1}+1)}\right).
 \] 
 Observe that for $1 \leq h \leq k-1$, 
      ${p_{h+1}}/({p_h+1}) \geq 1$, 
with equality if and only if  $p_h = 2$ and $p_{h+1}=3$.
Thus 
      ${n}/{Q} \geq \phi({n}/{p^{\alpha}}) p^{\alpha-1}(p-1){p_1}/({p+1})$, 
with equality if and only if $k=2$, $p_1=2$ and $p=3$. 
Since $p_1\geq 2$ and $(p-1)/({p+1})\geq 1/2$,
     ${p_{1}(p-1)}/({p+1}) \geq 1$, 
with equality if and only if $p_1=2$ and $p=3$.  
Thus  (\ref{eq:n=phinx11}) holds with equality if and only if 
$n=2^\alpha 3^\beta$ with $\alpha$, $\beta> 0$.
\end{proof}

\begin{lem}
\label{lem:ineq-gnotid}
With Notation \ref{nta:nfactored}, let $G$
be a finite group of order $n$, 
and let $g\in G$.  
If $n < Q\phi(\ord(g))$, then $g$ is not the identity of $G$ except possibly when $n= 2$.
\end{lem}

\begin{proof}
Suppose $g$ is the identity of $G$, so $\phi(\ord(g))=1$.
Observe that if $n=1$, then $Q=1$ (an empty product) and 
     $\phi(\ord(g))=1$. 
In this case 
      $n = Q\phi(\ord(g))$, 
so the lemma does not apply.  Assume $n\geq 2$. 
Lemma \ref{lem:n/Qgeqph(n/p)p} and the hypothesis imply that $n$ is a positive power of $2$.  In this case, $Q\phi(\ord(g))=3$, which is less than $n$ unless $n=2$.  When $n=2$, $n < Q\phi(\ord(g))$, so the exception is required.  
\end{proof}

\begin{lem}
\label{lem:N<Qphi-primpow}
With Notation \ref{nta:nfactored}, let $G$
be a finite group of prime power order $n>2$, 
and let $g\in G$.  
If $n < Q\phi(\ord(g))$,
then $g$ generates $G$. 
\end{lem}

\begin{proof}
Say  $n=p^\alpha$.  Then
        $Q=(p+1)/(p-1)$ by defintion, and
        $\ord(g)=p^\ell$ for some $\ell$ $(0< \ell\leq \alpha)$ by Lagrange's theorem and Lemma \ref{lem:ineq-gnotid}. 
Now           
     $\phi(\ord(g)) =p^{\ell-1}(p-1)$. 
Thus $Q\phi(\ord(g)) = p^{\ell-1}(p+1)$.    Now
      $p^\alpha=n < Q\phi(\ord(g))=p^{\ell-1}(p+1)$.  
Thus  
         $p^{\alpha-\ell+1}\leq p$,
so
          $\ell\geq \alpha$.
In addition       
         $\ell\leq \alpha$, 
so $\ell=\alpha$. Hence $g$ generates $G$.  
\end{proof}

\begin{lem}
\label{lem:generalres} 
With Notation \ref{nta:nfactored}, let $G$
be a finite group of order $n>2$, and let $g\in G$.  
If 
          $n < Q\phi(\ord(g))$,
then
          $p^{\alpha}|\ord(g)$. 
\end{lem}

\begin{proof}
If $n$ has just one prime factor, then  $g$ generates $G$ by Lemma \ref{lem:N<Qphi-primpow}, and the result follows.
Assume that $n$ has at least two distinct prime factors.
By hypothesis and Lemma \ref{lem:n/Qgeqph(n/p)p},
\begin{equation}
\label{eq:phiog>phi n/pp}
         \phi(\ord(g))   > \phi(\frac{n}{p^\alpha})    p^{\alpha-1}.
\end{equation}
For the sake of contradiction, suppose that 
      $p^{\alpha}\nmid \ord(g)$, 
so  $\ord(g)|n/p$.  
We consider two cases.  
If $\alpha = 1$, then (\ref{eq:a|b=>phia|phib}) gives 
    $\phi(\ord(g))| \phi(n/p)$,
contradicting (\ref{eq:phiog>phi n/pp}).  
If  $\alpha \geq 2$, then  (\ref{eq:a|b=>phia|phib}) gives
   $\phi(\ord(g))| \phi(n/p^{\alpha}) p^{\alpha-2}(p-1)$. 
In this case
   $\phi(\ord(g)) \leq \phi(n/p^{\alpha}) p^{\alpha-2}(p-1)$,
contradicting (\ref{eq:phiog>phi n/pp}).
We conclude that $p^\alpha | \ord(g)$, as required.
\end{proof}

\begin{lem}
\label{lem:N<Qphi,even}
With Notation \ref{nta:nfactored}, let $G$
be a finite group of order $n$, and let $g\in G$.  
If 
          $\ord(g)$ is even and 
          $n < Q\phi(\ord(g))$,
then
          $n/\ord(g) < p$. 
\end{lem}

\begin{proof}
Observe that 
      $\ord(g)\geq 2\phi(\ord(g))$ and 
      $p_1=2$, 
so 
     $n/\ord(g)\leq n/2\phi(\ord(g))\leq Q/2$.
If $n=2$, the result trivial.     
If $n=2^\alpha$ for some $\alpha>0$, then  $Q=3$ by definition and  
$\ord(g)=n$ by Lemma \ref{lem:N<Qphi-primpow}, so the result follows.    
Assume $n$ has at least one prime factor other than $2$. Then by 
(\ref{eq:sumphiineq-rhs}),
     $Q/2\leq  3(p+1)/2(p_2-1)$.
Since $p_2\geq 3$, the right-hand side is at most $p$, and the result follows.   
\end{proof}

\begin{defn}
\label{def:pcomp}
Let $p$ be a prime. Let $G$ be a finite group, and let $P$ be a Sylow
$p$-subgroup of $G$. A {\em $p$-complement} in $G$ is a  subgroup with index equal to the order of $P$.
\end{defn}

\begin{thm} 
\cite[Theorem 10.21]{r}  (Burnside's transfer theorem)
\label{thm:transfer}
With the notation of Definition \ref{def:pcomp},
if $P\subseteq Z(N_G(P))$, then $G$ has a normal $p$-complement.
\end{thm}

\begin{thm}
\label{thm:overall} 
With Notation \ref{nta:nfactored}, let $G$ be a finite group of order $n$. 
Suppose that there exists an element $g\in G$ such that 
    $n<Q\phi(\ord(g))$. Then there is a
normal (and hence unique) Sylow $p$-subgroup of $G$.  Moreover, the Sylow $p$-subgroup is contained in $\langle g\rangle$ and hence is cyclic.
\end{thm}

\begin{proof}
Note that if $n$ is a prime power,  then the result follows from Lemma \ref{lem:N<Qphi-primpow}, so ssume that $n$ is not a prime power.
First  suppose $n/\ord(g)< p+1$.  Then  
           $|G:\langle g\rangle|=  n/\ord(g) < p+1$.
By Lemma \ref{lem:generalres}, $p^{\alpha} | \ord(g)$, so         
$p\nmid |G:\langle g\rangle|$.  Thus $\langle g\rangle$ contains a Sylow $p$-subgroup $P$ of $G$ (which is necessarily cyclic since $\langle g\rangle$ is).   Clearly $\langle g\rangle\subseteq C_G(P)\subseteq N_G(P)$, so $|G:N_G(P)|< p+1$.  But $|G:N_G(P)|$ is the number of Sylow $p$-subgroups  and must be congruent to $1$ modulo $p$.  Thus, it must be the case that there is exactly one Sylow $p$-subgroup, which is necessarily normal. 

Now suppose $n/\ord(g)\geq p+1$.  Note that $n$ is not a power of 2, so Lemma \ref{lem:n/Qgeqph(n/p)p} gives
   $Q\leq n<Q\phi(\ord(g))$.
In particular, $\phi(\ord(g))>1$, so $\ord(g)>\phi(\ord(g))$.
Now  $n/\ord(g) \leq n/\phi(\ord(g)) < Q$. 
Thus by Lemmas \ref{lem:Q<=p} and \ref{lem:N<Qphi,even}, the following hold:
       $2\leq k\leq 8$,
       $n=\prod_{i=1}^{k}\pi(i)^{\alpha_i}$
       with $\alpha_i\not=0$ $(1\leq i \leq k)$,
       and $\ord(g)$ is odd.       
In Table \ref{tab:exceptionaln}, we show that other than $n=2\cdot 3\cdot 5^{\alpha}$, none of the remaining cases satisfy $n/\phi(\ord(g))<Q$, and thus are not subject to this theorem.  In this table, for $2\leq k \leq 8$ we mark with a bullet ($\bullet$) the even integers that are at least $\pi(k)+1$ and strictly less than $Q$ (from Table \ref{tab:specialQ}) as the possible values of  of $n/\ord(g)$
   Also by Lemma \ref{lem:generalres}, 
      $\pi(k)^{\alpha_k} | \ord(g)$, so 
      $\pi(k)\nmid  n/\ord(g)$. 
Since $\ord(g)$ is odd, $2^{\alpha_1} | n/\ord(g)$, where $\alpha_1$ is the largest power of 2 dividing $n/\ord(g)$.  It is now easy to read $\ord(g)$. The value of $\phi(\ord(g))$ will depend upon which primes appear in $\ord(g)$, but otherwise is straightforward to compute. All case other than
$n=2\cdot 3\cdot 5^{\alpha}$  violate  $n/\phi(\ord(g))<Q$.

Suppose $n=2\cdot 3\cdot 5^{\alpha}$.  Observe that $\ord(g)=5^\alpha$, so $\langle g \rangle$ is a cyclic Sylow $5$-subgroup. Note that the Sylow 2-subgroups are cyclic, so they are contained in the center of their normalizer.  Thus by Theorem \ref{thm:transfer}, there is a normal $2$-complement $H$  in $G$.  Now $H$ has order $3\cdot 5^\alpha$, its sylow $3$ subgroups are likewise cyclic, so there is a normal $3$-complement $P$ in $H$.  Now $P$ is a normal Sylow 5-subgroup of $H$, so it is characteristic in $H$, and hence normal in $G$.  Since $P$ is the unique Sylow 5-subgroup of $G$, we have 
$P=  \langle g \rangle$.  Thus the theorem holds in this case.
\end{proof}

\begin{table}
\[
\begin{array}{|rrrrr|}
\hline
k&\pi(k)&Q&&\\
  &\bullet\  \frac{n}{\ord(g)} & \alpha_1& \ord(g)&\\
  &  & \hbox{case} & \phi(\ord(g)) & \lfloor \frac{n}{\phi(\ord(g))}\rfloor \\
  \hline
2&3&6&&\\
  &\bullet\  \phantom{1}4 & 2& 3^{\alpha_1}&\\
  & & \hbox{all} & 2\cdot3^{\alpha_1-1} & 6= Q\\ 
\hline  
3&5&9&&\\
  &\bullet\  \phantom{1}6 & 1 & 3^{\alpha_2-1} 5^{\alpha_3}&\\
  & & \alpha_2=1 & 4\cdot 5^{\alpha_3-1}
            & \mathbf{7.4<Q}\\
  &  & \alpha_2>1 & 2\cdot 3^{\alpha_2-1}4\cdot 5^{\alpha_3-1} 
           & 11>Q\\
  &\bullet\  \phantom{1}8 & 3&3^{\alpha_2} 5^{\alpha_3}&\\
  & & \hbox{all} &2\cdot 3^{\alpha_2-1}4\cdot 5^{\alpha_3-1}
           & 15>Q\\
\hline  
4&7&12&&\\
  &\bullet\  \phantom{1}8 & 3&3^{\alpha_2} 5^{\alpha_3} 7^{\alpha_4}&\\
    & & \hbox{all} &2\cdot 3^{\alpha_2-1}4\cdot 5^{\alpha_3-1}
                             6\cdot 7^{\alpha_4}
                 & 17>Q\\
  &\bullet\  10& 1&3^{\alpha_2} 5^{\alpha_3-1} 7^{\alpha_4}&\\
    & & \alpha_3=1 & 2\cdot 3^{\alpha_2-1} 6\cdot 7^{\alpha_4}
                  & 14>Q\\
  &  & \alpha_3>1 & 2\cdot 3^{\alpha_2-1}4\cdot 5^{\alpha_3-2}
                             6\cdot 7^{\alpha_4}
                   &  21>Q\\
\hline  
5&11&14.4&&\\
  &\bullet\  12 & 2&3^{\alpha_2-1} 5^{\alpha_3} 
                                      7^{\alpha_4}11^{\alpha_5}&\\
  & & \alpha_2=1 & 4\cdot 5^{\alpha_3-1}6\cdot 
                              7^{\alpha_4}10\cdot 11^{\alpha_5-1}
                   & 19>Q\\
  &  & \alpha_2>1 & 2\cdot 3^{\alpha_2-1}4\cdot 5^{\alpha_3-1}
                               6\cdot 7^{\alpha_4}10\cdot 11^{\alpha_5-1}
                   &  28>Q\\ 
  &\bullet\  14 & 1&3^{\alpha_2} 5^{\alpha_3}
                                 7^{\alpha_4-1}11^{\alpha_5}&\\
    & & \alpha_4=1 & 2\cdot 3^{\alpha_2-1} 4\cdot 5^{\alpha_3-1}
                                   10\cdot 11^{\alpha_5-1}
                   & 28>Q\\ 
  &  & \alpha_4>1 & 2\cdot 3^{\alpha_2-1}4\cdot 5^{\alpha_3-2}
                             6\cdot 7^{\alpha_4}10\cdot 11^{\alpha_5-1}
                   & 33>Q\\    
\hline  
6&13&16.8&&\\
  &\bullet\  14 & 1&3^{\alpha_2} 5^{\alpha_3} 7^{\alpha_4-1}
                  11^{\alpha_5}13^{\alpha_6}&\\
  & & \alpha_4=1 & 2\cdot 3^{\alpha_2-1} 4\cdot 5^{\alpha_3-1}
                                   10\cdot 11^{\alpha_5-1}12\cdot13^{\alpha_6-1}
              &62>Q\\
  &  & \alpha_4>1 &  \left\{\begin{array}{l}
                          2\cdot 3^{\alpha_2-1}4\cdot 5^{\alpha_3-2}
                             6\cdot 7^{\alpha_4-2}\\
                  {} \times{}10\cdot 11^{\alpha_5-1}
                             12\cdot13^{\alpha_6-1}
                               \end{array}\right.
                & 36>Q\\  %
  &\bullet\  16 & 4&3^{\alpha_2} 5^{\alpha_3} 7^{\alpha_4}
                  11^{\alpha_5}13^{\alpha_6}&\\
  &&\hbox{all}&   \left\{\begin{array}{l}
                    2\cdot 3^{\alpha_2-1}4\cdot 5^{\alpha_3-2}
                             6\cdot 7^{\alpha_4-1}\\
                  {} \times{}10\cdot 11^{\alpha_5-1}
                             12\cdot13^{\alpha_6-1}  \end{array}\right.
                & 41>Q\\      
\hline  
7&17&18.9&&\\
  &\bullet\ 18 & 1&3^{\alpha_2-2} 5^{\alpha_3} 7^{\alpha_4}
                  11^{\alpha_5}13^{\alpha_6}17^{\alpha_7}&\\
    & & \alpha_2=2 & \left\{\begin{array}{l}
                   4\cdot 5^{\alpha_3-1}6\cdot 7^{\alpha_4-1}
                                   10\cdot 11^{\alpha_5-1}\\
                  {} \times{}12\cdot 13^{\alpha_6-1} 16\cdot 17^{\alpha_7-1}
                    \end{array}\right.
                   & 33>Q\\
  &  & \alpha_2>2 & \left\{\begin{array}{l}
                           2\cdot 3^{\alpha_2-1}4\cdot 5^{\alpha_3-2}
                             6\cdot 7^{\alpha_4-1}\\
                             {}\times 10\cdot 11^{\alpha_5-1}
                              12\cdot 13^{\alpha_6-1} 16\cdot 17^{\alpha_7-1}
                                \end{array}\right.
                   & 49>Q\\
\hline  
8&19& 21&&\\
  &\bullet\  20 & 2&3^{\alpha_2} 5^{\alpha_3-1} 7^{\alpha_4}
                  11^{\alpha_5}13^{\alpha_6}17^{\alpha_7}19^{\alpha_8}&\\
    & & \alpha_3=1 &\left\{\begin{array}{l}
                                    2\cdot 3^{\alpha_2-1}6\cdot 7^{\alpha_4-1}
                                   10\cdot 11^{\alpha_5-1} \\
                                   {}\times
                                   12\cdot 13^{\alpha_6-1}16\cdot 17^{\alpha_7-1}
                                    18\cdot 19^{\alpha_8-1}
                                \end{array}\right.
                   & 46>Q \\
  &  & \alpha_3>1 &\left\{\begin{array}{l}  
                             2\cdot 3^{\alpha_2-1}4\cdot 5^{\alpha_3-2}
                             6\cdot 7^{\alpha_4-1}10\cdot 11^{\alpha_5-1}\\ 
                              {}\times 12\cdot 13^{\alpha_6-1} 
                              16\cdot 17^{\alpha_7-1}  18\cdot 19^{\alpha_8-1}
                              \end{array}\right.
                   & 58>Q \\

\hline  
\end{array}
\]
\caption{Exceptional cases in the proof of Theorem \ref{thm:overall} }
\label{tab:exceptionaln}
\end{table}

The contrapositive form of Theorem \ref{thm:overall} is interesting.

\begin{cor}
With Notation \ref{nta:nfactored}, let
$G$ be a finite group of order $n$, and let $p$ be the largest prime divisor of $n$.
If there is more than one Sylow $p$-subgroup, then 
     $n\geq Q\phi(\ord(g))$
 for all $g\in G$.
\end{cor}

The bound in Theorem \ref{thm:overall} is tight in the following sense.
In the alternating group $\mathbb{A}_4$, $n=12$,
$Q= 6$, and elements have order $3$, $2$, and $1$.
For $g\in \mathbb{A}_4$ with $\ord(g)= 3$, $\phi(\ord(g))= 2$.  
Thus $n=Q\phi(\ord(g))$.  However, $\mathbb{A}_4$ has four  Sylow $3$-subgroups, which happen to be cyclic.

\section{Proof of the main theorem}

To prove Theorem \ref{thm:main}, we need some facts about direct and semi-direct products.

\begin{lem}
\label{eq:prodphidirect} 
Let $U$ and $T$ be finite groups, and let $G = U\times T$ be the direct product of $U$ and $T$. Then $\phi(G) \leq \phi(U)\phi(T)$. Moreover, if 
$(|U| , |T|) = 1$, then $\phi(G) = \phi(U)\phi(T)$.
\end{lem}

\begin{proof} 
Given $g=(u,t)\in G$, 
    $\ord(g)=\ord(u)\ord(t)/(\ord(u),\ord(t))$.
Thus by the multiplicative property of the totient function and by (\ref{eq:a|b=>phia|phib})
\[ \phi(\ord(g))=\phi(\frac{\ord(u)}{(\ord(u),\ord(t))})\phi(\ord(t))
                       \leq \phi(\ord(u))\phi(\ord(t)).\]
Now
\begin{equation}
\label{eq:phidirprod}
\begin{array}{rcl}
\phi(G) &=& \displaystyle{
                     \sum_{u\in U} \sum_{t\in T} \phi(\ord{(u,t)}) 
               = \sum_{u\in U} 
                     \sum_{t\in T}\phi(\frac{\ord(u)}{(\ord(u),\ord(t))}) \phi(\ord(t))}\\
 &\leq& \displaystyle{
             \sum_{u\in U} \phi(\ord(u))\sum_{t\in T}\phi(\ord(t))=  \phi(U)\phi(T).}
 \end{array}
\end{equation}
Observe that if $(|U| , |T|) = 1$, then $(\ord(u),\ord(v))=1$ for all $u\in U$ and $t\in T$, so equality holds throughout.
\end{proof}

The condition $(|U| , |T|) = 1$ in Lemma \ref{eq:prodphidirect} can be replaced with other conditions to reach the same conclusion.  
If $U$ is an elementary abelian 2-group, then all elements of $U$ have order 1 or  2.  The totient of these numbers and their divisors is 1,
 so 
     $\phi(\ord(u))=\phi(\ord(u)/(\ord(u),\ord(t)))=1$ for all $u\in U$ and $t\in T$.  
Now (\ref{eq:phidirprod}) gives  $\phi(G) = \phi(U)\phi(T)$.  
Similarly, if $(|U| , |T|) = 2$ and $|U|$ is twice an odd number, then 
     $\phi(\ord(u))=\phi(\ord(u)/(\ord(u),\ord(t)))$, 
so $\phi(G) = \phi(U)\phi(T)$.

\begin{lem}
\cite[Lemma 5.3]{BCGRP:maxedge}
\label{lem:divideorder1}
Suppose that $G$ is a finite group and
that $G=U\rtimes_\varphi V$ is the  semidirect product of a normal
abelian subgroup $U$ and a subgroup $V$.   Assume $U$ and  $V$
have coprime orders. Then 
     $\ord_G(uv)|\ord_{U\times V}(uv)$
 for all $u\in U$ and $v \in V$.
\end{lem}

\begin{cor}
\label{cor:sdp-div-dp}
With reference to Lemma \ref{lem:divideorder1},
     $\phi(\ord_G(uv))|\phi(\ord_{U\times V}(uv))$, and  
     $\phi(U\rtimes_\varphi V)\leq \phi(U\times V)$.
\end{cor}

\begin{proof}
The divisibility follows from Lemma \ref{lem:divideorder1} and (\ref{eq:a|b=>phia|phib}), and the inequality follows from (\ref{eq:newsumgroup}).
\end{proof}

\begin{thm}
\cite[Theorem 10.30]{r} 
(The Schur-Zassenhaus theorem)
\label{thm:schurzass} 
Let $G$ be a finite group, and let $K$ be a normal subgroup of  $G$ with $(|K|, |G:K|) = 1$. Then $G$ is a semidirect product of $K$ and $G/K$.  In particular, there exists a subgroup $H$ of $G$ with order $|G : K|$ such that $G = K\rtimes_\varphi H$  for some homomorphism $\varphi:H\rightarrow \mathrm{Aut}(K)$.
\end{thm}

Before treating the general case we present a special case involving  cyclic groups.

\begin{lem}
\label{lem:semivsdircyclics}
Let $a$ and $b$ be coprime positive integers.  Then
 $\phi(\cyc_a\rtimes_\varphi \cyc_b) < \phi(\cyc_a\times \cyc_b)$,
with equality if and only if the semi-direct product is direct.
\end{lem}

\begin{proof}
Note that 
     $G=\cyc_a\rtimes_\varphi \cyc_b$ and 
     $H=\cyc_a\times \cyc_b\cong \cyc_{ab}$
are defined on the cartesian product of the underlying sets of 
     $\cyc_a$ and $\cyc_b$.  
Let $n=ab$.     
By Corollary \ref{cor:sdp-div-dp}, 
   $\phi(\ord_G(g)) |  \phi(\ord_H(g))$ for all $g\in G$.
Thus 
 $\sum_{g\in G} \phi(\ord_G(g))
                 \leq  \sum_{g\in G} \phi(\ord_H(g))$.
Moreover, equality holds if and only if 
   $\phi(\ord_G(g))= \phi(\ord_H(g))$ for all $g\in G$
 
Suppose equality holds for the sums.  Pick a generator $h$ of $H$.  We are done if $\ord_G(h)=n$ since $G\cong\cyc_{n}\cong H$ in this case. Suppose for the sake of contradiction that $\ord_G(h)\neq n$.  Now $\ord_G(h)|n$ by  (\ref{eq:a|b=>phia|phib}), so in light of (\ref{eq:phi(n)primes}), 
     $m=\ord_G(h) = n/2$ is odd, as.
 Let $L=\langle h\rangle\subset G$, so $|L|$ is odd and $|G:H|=2$.  This implies $L\lhd G$.   Now by Theorem \ref{thm:schurzass}, there is a subgroup $K$ of $G$ with order 2 such that $G=L\rtimes_\psi K$.  Hence $G$ is isomorphic to the semi-direct product $\cyc_m\rtimes_\psi \cyc_2$.        
Since $\cyc_m$ is normal in $G$, we have that
    $(uv)^2 \in\cyc_m$   for all $u\in \cyc_m$, $v\in \cyc_2$.   
In particular, $\ord_G(uv)$ is even. However, 
    $\ord_G(uv)\not=2m$ 
since $G$ is not cyclic.  Now  
     $\phi(\ord_G(uv))<\phi(2m)=\phi(n)$,
since $\ord(u)|m$.     
This implies
   $\phi(G) < \phi(\cyc_n)$, contrary to our assumption.
Thus $G$ is cyclic as required.   
\end{proof}

We are ready to prove our main result, namely that  $\phi(\cyc_n)\geq \phi(G)$, with equality if and only if $G$ is isomorphic to $\cyc_n$.

\begin{proof}[Proof of Theorem \ref{thm:main-restated}]
Suppose 
       $\phi(G)\geq \phi(\cyc_n)$.
For some $g\in G$, $\phi(\ord(g))$ is at least the average value over the group, so 
    $\phi(\ord(g)) \geq \phi(G)/n\geq    \phi(\cyc_n)/n> n/Q$
by (\ref{eq:phiC>n^2/Q}).

We proceed by induction on the number of distinct prime factors of $n$.    
If $|G|$ has just one prime factor, then  $G$ is cyclic by Lemma \ref{lem:N<Qphi-primpow}, and hence isomorphic to $\cyc_n$.       
Now assume that for  all $n'$ with fewer distinct prime factors than $n$
and groups $G'$ of order $n'$,
      $\phi(\cyc_{n'}) \geq \phi(G')$,
with equality if and only if $G'$ is isomorphic to $\cyc_{n'}$.  

By Theorem \ref{thm:overall}, there exists a Sylow $p$-subgroup $P$ of $G$ which is both cyclic and normal, where $p$ is the largest prime divisor of $n$. Since $P$ is a Sylow $p$-subgroup,  $|G:P|$ is coprime to $|P|$.  Abbreviate $a=|P|$, $b=|G:P|$.  By Theorem \ref{thm:schurzass},
        $G=P\rtimes_\varphi T$
for some subgroup $T\subseteq G$ with order $b$ and some homomorphism $\varphi:T\rightarrow\mathrm{Aut}(P)$.  

Since $P$ is cyclic, Corollary \ref{cor:sdp-div-dp} gives that
         $\phi(G)=\phi( P\rtimes_\varphi T) \leq \phi(P\times T)$.
But by Lemma \ref{eq:prodphidirect},  
         $\phi(P\times T)= \phi(P)\phi(T)$.         
Identify $\cyc_n$ with the direct product of cyclic subgroups 
          $\cyc_a\times \cyc_b$.          
Observe that
        $\phi(\cyc_n)=\phi(\cyc_a)\phi(\cyc_b)$ 
by Lemma \ref{eq:prodphidirect} and 
        $\phi(\cyc_a)=\phi(P)$ 
since both are cyclic and of the same order.

Note that $p\nmid |T|=b$ by construction and $|T||n$ by Lagrange's theorem, so $|T|$ has fewer distinct prime divisors than $n$ and $|T|<n$. By the inductive hypothesis 
       $\phi(\cyc_b)\geq \phi(T)$,
with equality if and only if $T$ is cyclic.  Thus 
        $\phi(G)\leq \phi(\cyc_n)$,
with equality only if $T$ is cyclic.  
By assumption $\phi(G)\geq \phi(\cyc_n)$, hence,
        $\phi(G)= \phi(\cyc_n)$ and 
        $T$ is cyclic of order $b$.
Thus $G$ is isomorphic to $\cyc_a\rtimes_\varphi \cyc_b$.  
The result follows by Lemma \ref{lem:semivsdircyclics}.  
\end{proof}

\begin{proof}[proof of Theorem \ref{thm:main-restated}]
Straightforward from  Theorem \ref{thm:main-restated} and (\ref{eq:biEcount}).
\end{proof}

Theorem \ref{thm:main-restated} implies that 
 $\cyc_n$ is determined up to isomorphism by 
$\phi(\cyc_n)$.  However, $\phi(G)$ depends only upon the orders of its elements, and does not determine $G$ in general. 
Indeed,  
      $\phi(\cyc_4 \times \cyc_4)=\phi(\cyc_2 \times Q)=28$, 
 where $Q$ is the quaternion group, since each has three elements of order $2$ and twelve of order $4$.    We pose a related question. 
Let G and H are finite groups of the same order with 
       $\phi(G) = \phi(H)$. 
Suppose G be simple. Is H necessarily simple?


\end{document}